\documentclass{article}
\usepackage{comment}

\usepackage[backend=biber, style=numeric, sorting=none]{biblatex} 
\usepackage{subcaption}

\usepackage{graphicx}
\usepackage{enumitem}

\usepackage{xcolor}
\usepackage{hyperref}

\usepackage{amsthm}
\usepackage{amsmath}
\usepackage{amsfonts}
\usepackage{amssymb}
\usepackage{euscript}
\usepackage{bm}

\newcommand{\bA}{\mathbf{A}}
\newcommand{\bB}{\mathbf{B}}
\newcommand{\bI}{\mathbf{I}}

\newcommand{\bP}{\mathbf{P}}

\newcommand{\bX}{\mathbf{X}}
\newcommand{\bY}{\mathbf{Y}}

\newcommand{\sB}{\EuScript{B}}
\newcommand{\sC}{\EuScript{C}}
\newcommand{\sH}{\EuScript{H}}
\newcommand{\sL}{\EuScript{L}}
\newcommand{\sN}{\EuScript{N}}

\newcommand{\sP}{\EuScript{P}}
\newcommand{\sR}{\EuScript{R}}
\newcommand{\sS}{\EuScript{S}}

\newcommand{\sV}{\EuScript{V}}
\newcommand{\sX}{\EuScript{X}}

\newcommand{\bbR}{\mathbb{R}}
\newcommand{\bbC}{\mathbb{C}}
\newcommand{\bbF}{\mathbb{F}}

\newcommand{\bzero}{\mathbf{0}}

\newcommand{\bH}{\mathbf{H}}

\newcommand{\bbE}{\mathbb{E}}

\DeclareMathOperator{\Tr}{Tr}

\DeclareMathOperator{\rank}{rank}

\DeclareMathOperator{\Span}{Span}

\DeclareMathOperator{\Ext}{Ext}

\theoremstyle{plain}
\newtheorem{definition}{Definition}[section]

\newtheorem{theorem}[definition]{Theorem}
\newtheorem{corollary}[definition]{Corollary}
\newtheorem{remark}[definition]{Remark}
\newtheorem{example}[definition]{Example}

\usepackage{multirow}
\usepackage{multicol}

\usepackage{caption} 
\definecolor{kartik}{rgb}{0.8,0, 0.2}

\title{\textbf{Extreme Points of Spectrahedra}}
\author{\textsc{Kartik G. Waghmare\footnote{\texttt{\color{blue} kartik$\cdot$waghmare@math$\cdot$ethz$\cdot$ch} \hfill Seminar for Statistics, ETHZ, Zurich, Switzerland} and Victor M. Panaretos\footnote{\texttt{\color{blue} victor$\cdot$panaretos@epfl$\cdot$ch} \hfill Institute of Mathematics, EPFL, Lausanne, Switzerland}}}

\date{}

\begin{document}

\maketitle

\begin{abstract}
    We consider the problem of characterizing extreme points of the convex set of positive linear operators on a possibly infinite-dimensional Hilbert space under linear constraints. We show that even perturbations of points in such sets admit what resembles a Douglas factorization. Using this result, we prove that an operator is extreme iff a corresponding set of linear operators is dense in the space of trace-class self-adjoint operators with range contained in the closure of the range of that operator. If the number of constraints is finite, we show that the extreme point must be of low-rank relative to the number of constraints and derive a purely rank-based characterization of the extreme points. In the finite-dimensional setting, our results lead to a remarkably simple characterization of the elliptope, that is, the set of correlation matrices, in terms of the Hadamard product which allows us to characterize the set of matrices which constitute the equality case of the Hadamard rank inequality when the involved matrices are equal and positive semi-definite. We illustrate the importance of our results using examples from statistics and quantum mechanics.\\
\end{abstract}



\tableofcontents

\section{Introduction}

Let $\sH$ be a separable (real or complex) Hilbert space of possibly infinite dimension. For a linear operator $\bA: \sH \to \sH$, we define the \emph{trace} of $\bA$ as $\Tr \bA = \sum_{i} \langle \bA e_{i}, e_{i} \rangle$ where $\{e_{i}\}$ is an orthonormal basis of $\sH$. Let $A$ be a set.  For every $\alpha \in A$, let $\bA_{\alpha}: \sH \to \sH$ be a bounded linear operator and $c_{\alpha}$ be a constant. Corresponding to $\{(\bA_{\alpha}, c_{\alpha})\}_{\alpha \in A}$, we define the \emph{spectrahedron} $\sC$ as the set of bounded positive linear operators $\bP: \sH \to \sH$ satisfying the linear constraints
\begin{equation*}
    \Tr \bA_{\alpha}\bP = c_{\alpha} \mbox{ for every } \alpha \in A.
\end{equation*}
We consider the problem of characterizing the extreme points of $\sC$. An operator $\bP \in \sC$ is said to be \emph{extreme} if it cannot be represented as a strict convex combination of two distinct points in $\sC$.

The extreme points of general convex sets are of fundamental importance in convex analysis and optimization theory. For compact convex sets in locally convex Hausdorff spaces, the Krein-Milman theorem states that the set is the closure of the convex hull of its extreme points, and Bauer's maximum principle states that the unique maximum of a continuous convex function on the set is attained at an extreme point \cite{aliprantis2006}. As a consequence, the extreme points of spectrahedra play an important role in semi-definite optimization problems. The same is true for optimization problems on convex sets, such as the unit ball in spectral norm, which are not spectrahedra themselves but admit a spectrahedral representation \cite{saunderson2015}. Extreme points of spectrahedra also appear in other contexts, because they express the idea of irrepresentability. For example, in multivariate statistics, an extreme covariance would correspond to a random vector which can't be represented as the mixture of other similarly constrained independent random vectors. Similarly, in quantum physics, an extreme density matrix represents an extreme state which can't be represented as the superposition of other states under the same constraints. 

The finite-dimensional version of the problem, which concerns positive semi-definite matrices, has naturally been a subject of significant attention in the literature on semi-definite optimization. Remarkably, it has been shown in this setting that the extreme points are always of low rank relative to the number of linear constraints \cite{pataki1998}. Exploiting this simple fact has led to significant gains in computational peformance in solving semi-definite optimization problems \cite{burer2005}. The special cases of extreme correlation matrices and extreme states of finite-state quantum systems in the context of quantum marginal problems have also been studied in the literature in linear algebra and quantum physics \cite{li1994, christensen1979,grone1990, parthasarathy2005, rudolph2004,chiribella2024}.

We are primarily concerned with the infinite-dimensional version of the problem which is relevant to infinite-dimensional semi-definite optimization problems but doesn't seem to have received attention in the extant literature with the notable exception of \cite{kiukas2008}.
Infinite-dimensional linear and conic programming has a long history \cite{duffin1956}. Our interest in infinite-dimensional semi-definite programs is driven by certain problems in the statistics of functional data and in quantum mechanics, where covariances of random functions and states of particles, respectively, naturally correspond to positive operators of finite trace on an infinite-dimensional Hilbert space \cite{hsing2015, sakurai2020}. For statistics, classical results in the finite-dimensional setting address only the case when the random element takes values in a finite-dimensional vector space, which is not the case for random elements corresponding to continuous-time stochastic processes for example \cite{hsing2015}. Similarly, for quantum mechanics, these results account only for finite-state quantum mechanical systems which deal with spins or polarizations, since the state of a particle naturally corresponds to a density operator in a finite-dimensional Hilbert space in these cases. They are not directly applicable to those dealing with continuous quantities such as position or momentum because the state is then represented by a density operator in an infinite-dimensional Hilbert space \cite{sakurai2020}. Using our results, one can address the case when the concerned quantity is continuous and bounded as is the case for the position of a particle in a box, for example. 

While our primary focus is on infinite-dimensional Hilbert spaces, our findings are also relevant in the finite-dimensional context, where somewhat similar results have been previously established before \cite{li1994}. However, we provide a more quantitative description of perturbations, which is crucial for extending the results to the infinite-dimensional case, and of interest even in finite dimensions. Our approach is inspired by \cite{waghmare2024}, where the extreme points of the set of positive-definite completions of kernels were characterized in terms of bounded linear functionals on a projective tensor product space.
Needless to say, we also establish the low-rankedness of extreme points in infinite dimensions when the number of constraints is finite and prove a rank-based characterization of the extreme points. 
Importantly, our results lead to a remarkably simple characterization of the extreme points of the elliptope (the set of correlation matrices) in terms of the Hadamard product which appears novel. 

The article is organized as follows: We begin by presenting the basic background required for working with operators in infinite-dimensional Hilbert spaces in Section \ref{sec:background}. This is followed by a more refined description of the problem and our results in Section \ref{sec:problem}. We discuss our characterization of the extreme points of the elliptope and the interesting connection to the Hadamard rank inequality in Section \ref{sec:hadamard}. 
Finally, we illustrate how our results are useful for solving some infinite-dimensional semi-definite optimization problems arising in statistics and quantum mechanics in Section \ref{sec:applications}.

\section{Background and Notation}
\label{sec:background}

For a convex subset $\sC$ of a vector space $\sV$, we say that $p \in \sC$ is an extreme point on $\sC$ if it cannot be represented as a strict convex combination of two distinct points in $\sC$, that is, there do not exist distinct $p_{1}, p_{2} \in \sC$ and $\alpha \in (0, 1)$ such that $p = \alpha p_{1} + (1-\alpha) p_{2}$. It is not difficult to see that $p \in \sC$ is an extreme point of $\sC$ if it does not admit a nonzero \emph{even perturbation}, that is, there does not exist a nonzero $h \in \sV$ such that $p + h, p-h \in \sC$. We denote the set of extreme points of $\sC$ by $\Ext \sC$.\\

Let $\sH$ be a separable Hilbert space on the field $\bbF = \bbR$ or $\bbC$. The \emph{adjoint} $\bA^{\ast}$ of a bounded linear operator $\bA: \sH \to \sH$ is the unique linear operator $\bA^{\ast}: \sH \to \sH$ which satisfies $\langle \bA x, y \rangle = \langle x, \bA^{\ast} y \rangle$ for every $x, y \in \sH$. We say that $\bA$ is \emph{self-adjoint} if $\bA^{\ast} = \bA$. A linear operator $\bP: \sH \to \sH$ is said to be positive if it is self-adjoint and satisfies $\langle x, \bP x \rangle \geq 0$ for every $x \in \sH$. An operator $\bA: \sH \to \sH$ is said to be \emph{compact} if for every sequence $\{x_{i}\}_{i=1}^{\infty} \subset \sH$, the image $\{ \bA x_{i} \}_{i=1}^{\infty}$ admits a convergent subsequence. For every $f, g \in \sH$, we define the linear operator $f \otimes g: \sH \to \sH$ as $(f \otimes g)x = \langle g, x \rangle f$ for every $x \in \sH$. The range $\{y \in \sH: y = \bA x \mbox{ for some } x \in \sH \}$ and null space $\{y \in \sH: \bA y = 0\}$ of $\bA$ are denoted as $\sR(\bA)$ and $\sN(\bA)$. The symbols $\bI$ and $\bzero$ denote the identity and zero operator, respectively. We denote the set of bounded, positive and compact operators on $\sH$ by $\sS$, $\sP$ and $\sS_{\infty}$ respectively. 

According to the spectral theorem for compact self-adjoint operators, every compact positive operator $\bP: \sH \to \sH$ can be represented as $\bP = \sum_{j=1}^{\infty} \lambda_{j} f_{j} \otimes f_{j}$ where $\{(\lambda_{j}, f_{j})\}_{j=1}^{\infty}$ are the eigenvalue-eigenvector pairs of $\bP$, $\lambda_{j} \geq 0$ for $j \geq 1$ and $\lim_{j \to \infty} \lambda_{j} = 0$. For every $\alpha > 0$, we define the power $\bP^{\alpha}$ as $\bP^{\alpha} = \sum_{j=1}^{\infty} \lambda_{j}^{\alpha} f_{j} \otimes f_{j}$.\\

Let $\{e_{i}\}_{i=1}^{\infty} \subset \sH$ be an orthonormal basis of $\sH$. We define the \emph{trace} $\Tr \bA$ of a bounded operator $\bA$ as $\Tr \bA = \sum_{j=1}^{\infty} \langle e_{j},  \bA e_{j} \rangle$ when the series converges. Note that trace has the cyclic property, that is, $\Tr \bA \bB = \Tr \bB \bA$ for $\bA, \bB \in \sS$ such that the trace of at least one of the products is defined. It can be shown that the particular choice of the orthonormal basis is immaterial to the definition of trace. Note that if $\bA$ is a compact operator, then $\bA^{\ast}\bA$ is compact and positive, and we can define $|\bA| = (\bA^{\ast}\bA)^{1/2}$. For $1 \leq p < \infty$, the \emph{trace ideal} $\sS_{p}$ is the Banach space
\begin{equation*}
    \sS_{p} = \textstyle \Big\{\bA: \sH \to \sH ~\Big|~ \bA \mbox{ is compact and } \sum_{j=1}^{\infty} \langle e_{j},  |\bA|^{p} e_{j} \rangle < \infty \Big\}
\end{equation*}
equipped with the norm $\|\bA\|_{p} = \Tr |\bA|^{p}$. The particular instances $\sS_{1}$ and $\sS_{2}$ are referred to as the space of trace-class operators and the space of Hilbert-Schmidt operators, respectively. Moreover, $\sS_{p}$ is indeed an ideal in the algebraic sense: for $\bA \in \sS_{p}$ and $\bX \in \sS$, we have $\bA\bX , \bX\bA \in \sS_{p}$. In fact, for $1 \leq p, q, r \leq \infty$ such that $1/r = 1/p+1/q$ (we consider $1/\infty$ as $0$), we have for $\bA \in \sS_{p}$ and $\bB \in \sS_{q}$, that $\bA\bB \in \sS_{r}$ and $\|\bA\bB\|_{r} \leq \|\bA\|_{p}\|\bB\|_{q}$. 

It can be shown that the trace $\Tr: \sS_{1} \to \bbF$ given by $\Tr \bA = \sum_{j=1}^{\infty} \langle e_{j},  \bA e_{j} \rangle$ is actually a well-defined bounded linear functional on $\sS_{1}$. Furthermore, for $1 < p, q < \infty$ such that $1/p + 1/q = 1$, every continuous linear functional $l: \sS_{p} \to \bbF$ can be represented by an operator $\bA \in \sS_{q}$ in the sense that
\begin{equation*}
    l(\bX) = \Tr \bA \bX
\end{equation*} 
for every $\bX \in \sS_{p}$. In other words, $\sS_{p}^{\ast} \equiv \sS_{q}$. For the excepted cases of $p = 1$ and $\infty$, we have $\sS_{1}^{\ast} = \sS$ and $\sS_{\infty}^{\ast} = \sS_{1}$ instead. A linear functional $l$ is said to be self-adjoint if $l(\bA) = l(\bA^{\ast})$ for every $\bA$ in its domain. The reader is invited to consult Section 3.6-3.7 of \cite{simon2015} for a more complete introduction to the subject of trace ideals.

\section{Problem Statement and Results}
\label{sec:problem}

Let $1 \leq p < \infty$ and $q$ be such that $1/p + 1/q = 1$ (we take $q = \infty$ for $p =1$). 
Let $A$ be a set, $\{\bA_{\alpha}\}_{\alpha \in A} \subset \sS_{p}$ and $\{c_{\alpha}\}_{\alpha \in A} \subset \bbF$. 
Define $\sP_{q} = \sP \cap \sS_{q}$. 
Consider the \emph{spectrahedron} $\sC$ given by 
\begin{equation}\label{eqn:linear-constraints}
    \sC = \Big\{ \bP \in \sP_{q} ~\Big|~ \Tr \bP\bA_{\alpha} = c_{\alpha} \mbox{ for every } \alpha \in A \Big\}
\end{equation}
and notice that $\sC$ is a closed convex subset of closed convex subset of $\sS_{q}$.
We are interested in deriving the necesary and sufficient conditions for $\bP \in \sC$ to be an extreme point of $\sC$. 
The operators $\bA_{\alpha}$ will be assumed to be self-adjoint and this comes without any loss of generality since $\Tr \bA^{\ast}\bP = \Tr \left[\tfrac{1}{2}(\bA + \bA^{\ast})\bP\right]$ for $\bP \succeq \bzero$ where $\tfrac{1}{2}(\bA + \bA^{\ast})$ is symmetric.
The set $\sP$ can be understood as a spectrahedron corresponding to the empty set $A = \varnothing$.

\subsection{Even Perturbations and Douglas Factorization}

We begin by considering a simpler modification of the problem which is when $A$ is empty and $\sC = \sP$. 
It is not difficult to see that the only extreme point is the zero operator since for every nonzero $\bP \in \sP$, we can write $\bP \pm \tfrac{1}{2}\bP \in \sP$. However, the key to understanding the general problem is to understand the structure of all possible even perturbations of $\bP \in \sP$ which can be described as a kind of Douglas factorization \cite{douglas1966}.

\begin{theorem}[Douglas Factorization]\label{thm:douglas-factorization}
    An element $\bH \in \sS$ is an even perturbation of $\bP \in \sP$ in $\sP$ iff 
    $$\bH = \sqrt{\bP}\bX\sqrt{\bP}$$ for some self-adjoint $\bX \in \sS$ such that $\sR(\bX) \subset \overline{\sR(\bP) }$ and $\|\bX\| \leq 1$.
\end{theorem}

\begin{proof}
    If $\bP \pm \bH \succeq \bzero$ then $\bH^{\ast} = \bH$ and $- \langle \bP x, x \rangle \leq \langle \bH x, x \rangle \leq \langle \bP x, x \rangle$ or
    \begin{equation*}
        |\langle \bH x, x \rangle| \leq \langle \bP x, x \rangle = \langle \sqrt{\bP} x, \sqrt{\bP} x \rangle = \|\sqrt{\bP}x \|^{2}
    \end{equation*}   
    for every $x \in \sH$. 
    Moreover, we can write for $x, y \in \sH$,
    \begin{align*}
        |\langle \bH x, y \rangle| &= \tfrac{1}{4}\left| \langle \bH (x+y), (x+y) \rangle - \langle \bH (x-y), (x-y) \rangle \right| \\
        &\leq \tfrac{1}{4}\left[ \left| \langle \bH (x+y), (x+y) \rangle \right| + \left|\langle \bH (x-y), (x-y) \rangle \right|\right] \\
        &\leq \tfrac{1}{4}\left[ \|\sqrt{\bP}(x+y) \|^{2} + \|\sqrt{\bP}(x-y) \|^{2}\right] \\
        &= \tfrac{1}{2}\left[ \|\sqrt{\bP}x \|^{2} + \|\sqrt{\bP}y \|^{2}\right].
    \end{align*}
    Replacing $x$ with $cx$ and $y$ with $y/c$ for some $c > 0$ gives 
    \begin{align*}
        |\langle \bH x, y \rangle| &\leq \tfrac{1}{2}\left[ c^{2}\|\sqrt{\bP}x \|^{2} + \frac{1}{c^{2}}\|\sqrt{\bP}y \|^{2}\right]
    \end{align*}
    and minimizing over $c$ yields $|\langle \bH x, y \rangle| \leq \|\sqrt{\bP}x \| \|\sqrt{\bP}y \|$. Notice that this implies that $\sN(\bH) \supset \sN(\bP)$.
    
    Define the bilinear functional $B: \sR(\sqrt{\bP}) \times \sR(\sqrt{\bP}) \to \bbF$ as $B(\sqrt{\bP}x, \sqrt{\bP}y) = \langle \bH x, y \rangle$. Then $B$ is well-defined and continuously extends to $\smash{\overline{\sR(\sqrt{\bP})} \times \overline{\sR(\sqrt{\bP})}}$ since $\smash{|B(\sqrt{\bP}x, \sqrt{\bP}y)| \leq \|\sqrt{\bP}x \| \|\sqrt{\bP}y \|}$. By the Riesz representation theorem, there exists a unique self-adjoint $\bX \in \sS$ such that $\smash{\sR(\bX) \subset \overline{\sR(\sqrt{\bP})} = \overline{\sR(\bP)}}$, $\|\bX\| \leq 1$ and $B(x, y) = \langle \bX x, y \rangle$ for every $\smash{x,y \in \overline{\sR(\sqrt{\bP})}}$.
    Furthermore, if we define $B(x, y) = 0$ if $x$ or $y \in \sN(\bP)$ then $B$ extends to $\sH \times \sH$ such that $|B(x, y)| \leq \|x\|\|y\|$ and $B(\sqrt{\bP}x, \sqrt{\bP}y) = \langle \bH x, y \rangle$ for $x,y \in \sH$.  It follows that $\bH = \sqrt{\bP}\bX\sqrt{\bP}$. The converse is immediate. 
\end{proof}

\begin{remark}\label{remark:Xfunctional}
    Note that by restricting $\bX$ to $\overline{\sR(\bP)}$, we can think of it as a bounded linear operator from $\overline{\sR(\bP)}$ to $\overline{\sR(\bP)}$. Because $\sS_{1}^{\ast} = \sS$ for any Hilbert space, we can think of $\bX: \overline{\sR(\bP)} \to \overline{\sR(\bP)}$ as a self-adjoint bounded linear functional $l_{\bX}: \sS_{1}(\overline{\sR(\bP)}) \to \bbF$ on the space $\sS_{1}(\overline{\sR(\bP)})$ of trace class linear operators $\bA: \overline{\sR(\bP)} \to \overline{\sR(\bP)}$ as $l_{\bX}(\bA) = \Tr \bX \bA$.
\end{remark}

Note that there is a one-one correspondence between the even perturbations $\bH$ of $\bP$ and the self-adjoint operators $\bX \in \sS$ satisfying $\sR(\bX) \subset \overline{\sR(\bP)}$ and $\|\bX\| \leq 1$ of Theorem \ref{thm:douglas-factorization}. With a slight abuse of terminology, we define the \emph{set of even perturbations} of $\bP$ in $\sC$ as \begin{center}$\sC_{\bP} = \{\bP + \bH: \bH \in \sS \mbox{ is an even perturbations of } \bP \mbox{ in } \sP \} \subset \sP.$\end{center} The set $\sC_{\bP}$ represents how much we can (evenly) move around $\bP$ in $\sS$ without escaping $\sC$. It is not difficult to see that $\sC_{\bP}$ can be parametrized in terms of the \emph{ball} $\sB_{\bP} = \{\bX \in \sS: \bX^{\ast} = \bX, \sR(\bX) \subset \overline{\sR(\bP)} \mbox{ and } \|\bX\| \leq 1 \}$ using the bijective linear mapping $\bA \mapsto \bP + \sqrt{\bP} \bA \sqrt{\bP}$ as $\sC_{\bP} = \bP + \sqrt{\bP} \sB_{\bP} \sqrt{\bP}$.

\subsection{General Characterization}

The insight into the structure of even perturbation provided by Theorem \ref{thm:douglas-factorization} can now be exploited to solve the problem in its generality. 

\begin{theorem}[General Characterization]\label{thm:lin-eq-char}
    Let $\sC$ be the spectrahedron in $\sP_{q}$ corresponding to $\{(\bA_{\alpha}, \beta_{\alpha})\}_{\alpha \in A}$ as in (\ref{eqn:linear-constraints}) and let $\bP \in \sC$. Then, the following three are equivalent:  
    \begin{enumerate}
\item $\bP \in \Ext \sC$
    
        \item there exists no nonzero self-adjoint $\bX \in \sS$ such that such that $\sR(\bX) \subset \overline{\sR(\bP) }$ and
        \begin{equation}\label{eqn:lin-eq-char}
            \Tr \sqrt{\bP}\bA_{\alpha} \sqrt{\bP}\bX = 0 \mbox{ for } \alpha \in A, 
        \end{equation}
        
        \item $\Span ~\{\sqrt{\bP}\bA_{\alpha} \sqrt{\bP}\}_{\alpha \in A}$ is dense in the subspace of self-adjoint trace-class operators with range in $\overline{\sR(\bP)}$.
    \end{enumerate}
\end{theorem}

\begin{remark}Curiously, the linear constraints $\Tr \bP \bA_{\alpha} = \beta_{\alpha}$ can be expressed in a form similar to \eqref{eqn:lin-eq-char}, namely as $\Tr \sqrt{\bP}\bA_{\alpha} \sqrt{\bP} = \beta_{\alpha}$, although this does not seem to have any deeper significance.
\end{remark}

\begin{proof}
    Let $\bH \in \sS$ be an even perturbation of $\bP$ in $\sC$, that is, $\bP \pm \bH \in \sC$. In other words, it is a perturbation of $\sP$, that is, $\bH^{\ast} = \bH$ and $\bP \pm \bH \succeq \bzero$, and satisfies $\Tr \bH\bA_{\alpha} = 0$ for $\alpha \in A$. By Theorem \ref{thm:douglas-factorization}, this is equivalent to saying that there exists some self-adjoint $\bX \in \sS$ such that $\sR(\bX) \subset \overline{\sR(\bP) }$ and
    \begin{equation}\label{eqn:papx}
        \Tr \sqrt{\bP} \bX  \sqrt{\bP}\bA_{\alpha} = \Tr \sqrt{\bP}\bA_{\alpha}\sqrt{\bP}\bX = 0 \mbox{ for } \alpha \in A,
    \end{equation} 
    since we can always scale $\bX$ such that $\|\bX\| \leq 1$. Note that $\sqrt{\bP}\bA_{\alpha}\sqrt{\bP} \in \sS_{1}$. It follows that $\bP$ is an extreme point iff such an $\bX$ does not exist.
    
    As in Remark \ref{remark:Xfunctional}, we can think of  $\bX$ as a bounded linear functional $l_{\bX}$ of norm $\leq 1$ on the space of trace-class operators $\bA: \overline{\sR(\bP)} \to \overline{\sR(\bP)}$. Notice that $l_{\bX}$ is necessarily a self-adjoint bounded extension of norm $\leq 1$ of the linear functional $l: \Span ~\{\sqrt{\bP}\bA_{\alpha}\sqrt{\bP}\}_{\alpha \in A}  \to \bbF$ given by $l_{\bX}(\bY) = \Tr \bY\bX$ to the whole space $\sS_{1}(\overline{\sR(\bP)})$ or equivalently, a bounded extension of norm $\leq 1$ to the subspace of self-adjoint operators in $\sS_{1}(\overline{\sR(\bP)})$. It follows that there exists no nonzero $\bX$ satisfying the required qualities iff the only bounded extension $l$ admits to the subspace of self-adjoint operators in $\sS_{1}(\overline{\sR(\bP)})$ is zero, which is precisely when the domain $\Span ~\{\sqrt{\bP}\bA_{\alpha}\sqrt{\bP}\}_{\alpha \in A}$ of $l$ is dense in the subspace of self-adjoint operators on $\sS_{1}(\overline{\sR(\bP)})$. Hence proved.
\end{proof}

Theorem \ref{thm:lin-eq-char} reduces the question of whether $\bP$ is extreme to that of whether a system of linear equations has a particular kind of nonzero solution. The set $\sC_{\bP}$ of even perturbations of $\bP$ in $\sC$ is now given by
\begin{equation*}
    \sC_{\bP} = \bP + \sqrt{\bP} \left[ \sB_{\bP} \cap \sL \right] \sqrt{\bP},
\end{equation*}
where the space $\sL = \{\bX \in \sS: \Tr \sqrt{\bP}\bA_{\alpha} \sqrt{\bP}\bX = 0 \mbox{ for } \alpha \in A\}$  ensures that the perturbations respect the linear constraints on $\bP$.

\subsection{Finite Number of Constraints}
The results acquire an interesting flavor when the number of constraints is finite. In particular, the extreme points have to be of finite rank and one can abstract away the connection to linear systems and express the result purely in terms of ranks.

\begin{theorem}[Rank Characterization]\label{thm:rank-eq-char}
    Let $n \geq 1$, $A = \{1, \dots, n\}$ and $\sC$ be the spectrahedron corresponding to $\{(\bA_{j}, \beta_{j})\}_{j=1}^{n}$ as in (\ref{eqn:linear-constraints}). For $\bP \in \sC$, we have that $\bP \in \Ext \sC$ iff the following three conditions all hold true:
     
    \begin{enumerate}
\item $\rank \bP < \infty$.
    
        \item $\rank \left[\Tr \bP \bA_{i}\bP \bA_{j}\right]_{i,j=1}^{n} = \tfrac{1}{2}(\rank \bP) (1 + \rank \bP)$ for $\bbF = \bbR$. 
        \item $\rank \left[\Tr \bP \bA_{i}\bP \bA_{j}\right]_{i,j=1}^{n} = (\rank \bP)^{2}$ for $\bbF = \bbC$.
    \end{enumerate}
\end{theorem}
\begin{proof}
    By Theorem \ref{thm:lin-eq-char}, $\bP \in \Ext \sC$ is equivalent to the absence of a nonzero self-adjoint operator $\bX \in \sS$ satisfying $\sR(\bX) \subset \overline{\sR(\bP)}$ and (\ref{eqn:lin-eq-char}). This is same as saying that the subspace 
    \begin{equation*}
        \sX(\bP) = \{\bX \in \sS: \bX^{\ast} = \bX \mbox{ and } \sR(\bX) \subset \overline{\sR(\bP)}\} \subset \sS.
    \end{equation*}
    is spanned by the operators $\{\sqrt{\bP}\bA_{j} \sqrt{\bP}\}_{j=1}^{n} \subset \sX(\bP)$, which is precisely when $\bP$ is of finite rank (as $\dim \sR(\bP) < \infty$) and
    \begin{equation*}
        \dim \sX(\bP) = \dim \Span \{\sqrt{\bP}\bA_{j} \sqrt{\bP}\}_{j=1}^{n}.
    \end{equation*}
    Here we can write
    \begin{align*}
        \dim \Span \{\sqrt{\bP}\bA_{j} \sqrt{\bP}\}_{j=1}^{n} 
        &= \rank \left[ \Tr \sqrt{\bP}\bA_{i} \sqrt{\bP} \cdot \sqrt{\bP}\bA_{j} \sqrt{\bP}  \right]_{i,j=1}^{n} \\
        &= \rank \left[\Tr \bP \bA_{i}\bP \bA_{j}\right]_{i,j=1}^{n}
    \end{align*}
    using the fact that the dimension of the span of vectors is equal to the rank of their Gram matrix. Moreover, $\dim \sX(\bP)$ is equal to the dimension of the space spanned by self-adjoint matrices $\bX \in \bbF^{r \times r}$ where $r = \rank \bP$ which can be calculated as $\tfrac{1}{2}r (1 + r)$ for $\bbF = \bbR$ and $r^{2}$ for $\bbF = \bbC$. The result follows.
\end{proof}

Theorem \ref{thm:rank-eq-char} thus allows us to describe the set $\Ext \sC$ in terms of a rank equation as the set of $\bP \in \sC$ which satisfy
\begin{equation*}
    \rank \left[\Tr \bP \bA_{i}\bP \bA_{j}\right]_{i,j=1}^{n} = \dim \sX(\bP) 
\end{equation*}
where $\dim \sX(\bP) = \tfrac{1}{2}(\rank \bP) (1 + \rank \bP)$ for $\bbF = \bbR$ and $(\rank \bP)^{2}$ for $\bbF = \bbC$.

Using Theorem \ref{thm:rank-eq-char}, we can completely work out the set of rank one extreme points of a spectrahedron.
\begin{corollary}[Rank One Extreme Points]
    Let $\sC$ be the spectrahedron corresponding to $\{(\bA_{j}, \beta_{j})\}_{j=1}^{n}$. Then for both $\bbF = \bbR$ and $\bbC$,
    \begin{enumerate}
        \item if $\{\beta_{j}\}_{j=1}^{n} = \{0\}$, then $\sC$ has no rank one extreme points, and
        \item if $\{\beta_{j}\}_{j=1}^{n} \neq \{0\}$, then the rank one extreme points of $\sC$ are given precisely by $\bP = x \otimes x$ for nonzero $x \in \sH$ satisfying $\langle \bA_{j}x, x \rangle = \beta_{j}$ for every $j \in \{1, \dots, n\}$.
    \end{enumerate}
\end{corollary}
\noindent
In other words, if every $\beta_{j}$ is $0$, then $\sC$ has no extreme points and if $\beta_{j}$ are not all zero, then the rank one extreme operators of $\sC$ are precisely the rank one operators in $\sS$ satisfying the linear constraints (\ref{eqn:linear-constraints}). The following example further illustrates the utility of Theorem \ref{thm:rank-eq-char}.

\begin{example}[Density Operators]    
    The spectrahedron $\sC = \{\bP \succeq \bzero: \Tr \bP = 1\}$ corresponding to density operators in $\sS$ is given by $\{(\bA_{j}, \beta_{j})\}_{j=1}^{n} =\{(\bI, 1)\}$. Consequently, $\bP \in \sC$ is an extreme density operator iff 
    \begin{equation*}
        \rank \left[\Tr \bP \bI \bP \bI \right]_{i,j=1}^{1} = 1 = (\rank \bP)^{2},
    \end{equation*}
    that is, $\rank \bP = 1$. It follows that $\Ext \sC = \{x \otimes x: x \in \sH \mbox{ and } \|x\|_{\sH} = 1 \}$.
\end{example}

Theorem \ref{thm:rank-eq-char} also allows us to derive the classical Barvinok-Pataki bounds on the ranks of extreme points. In semi-definite programming and quantum physics, often we do not know if a given semi-definite program or quantum marginal problem admits a solution (that is, whether $\sC$ is nonempty). In these cases, such upper bounds tell us that a solution exists iff there exists a solution of low enough rank and this often simplifies the problem computationally.

\begin{corollary}\label{thm:bpbounds}
    Let $n \geq 1$, $A = \{1, \dots, n\}$ and $\sC$ be the spectrahedron corresponding to $\{(\bA_{j}, \beta_{j})\}_{j=1}^{n}$ as in (\ref{eqn:linear-constraints}). If $\bP \in \Ext \sC$, then $n \geq \tfrac{1}{2}(\rank \bP) (1 + \rank \bP)$ for $\bbF = \bbR$ and $n \geq (\rank \bP)^{2}$ for $\bbF = \bbC$.
\end{corollary}

Even when $\bP \in \sC$ is not an extreme point, the dimension of the subspace spanned by its even perturbation tells us something interesting about the location of $\bP$ in $\sC$. This is known as the \emph{facial dimension} of $\bP$ in $\sC$ and it is not difficult to see that we can write it down explicitly as  
\begin{equation*}
    \dim \sX(\bP) - \rank \left[\Tr \bP \bA_{i}\bP \bA_{j}\right]_{i,j=1}^{n}
\end{equation*}
where $\dim \sX(\bP) = \tfrac{1}{2}(\rank \bP) (1 + \rank \bP)$ for $\bbF = \bbR$ and $(\rank \bP)^{2}$ for $\bbF = \bbC$.

\section{Correlation Matrices and Hadamard Rank Inequality}
\label{sec:hadamard}

Let $n \geq 1$, $\sH = \bbF^{n}$ and $\{e_{i}\}_{i=1}^{n}$ be the standard orthonormal basis of $\sH$. The space $\sS$ can then be identified with the space $\bbF^{n \times n}$ of $n \times n$ matrices with entries in $\bbF$. We say that $\bP \in \sP$ is a correlation matrix if $\bP_{ii} = \langle e_{i}, \bP e_{i} \rangle = 1$ for $1 \leq i \leq n$. The spectrahedron $\sC$ of correlation matrices is given by $\{(\bA_{j}, \beta_{j})\}_{j=1}^{n} =\{(e_{j} \otimes e_{j}, 1)\}_{j=1}^{n}$ as in (\ref{eqn:linear-constraints}). Using our results, we can now produce an arguably neat characterization of $\Ext \sC$.

For $\bA = [\bA_{ij}]_{i,j=1}^{n}, \bB  = [\bB_{ij}]_{i,j=1}^{n} \in \bbF^{n \times n}$, we define the Hadamard product $\bA \odot \bB$ as the matrix $[\bA_{ij}\bB_{ij}]_{i,j=1}^{n}$ of entry-wise products.
Naturally, the Hadamard square $\bA^{\odot 2}$ of $\bA$ is given by $\bA^{\odot 2} = \bA \odot \bA$.

\begin{corollary}\label{thm:extrm-correlation}
    Let $\sC$ be the convex set of correlation matrices in $\bbF^{n \times n}$. Then for $\bP \in \sC$,  
    \begin{enumerate}
        \item When $\bbF=\bbR$, we have  $\bP \in \Ext \sC \iff \rank \bP^{\odot 2} = \tfrac{1}{2}(\rank \bP) (1 + \rank \bP)$ 
        \item When $\bbF=\bbC$, we have  $\bP \in \Ext \sC \iff\rank \bP^{\odot 2} = (\rank \bP)^{2}$ 
    \end{enumerate}
\end{corollary}
\begin{proof}
    Note that $\left[\Tr \bP \bA_{i}\bP \bA_{j}\right]_{i,j=1}^{n} = [\bP_{ij}^{2}]_{i,j = 1}^{n} = \bP^{\odot 2}$ which is the Hadamard square of the matrix $\bP$, since
    \begin{equation*}
        \Tr \bP \bA_{i} \bP \bA_{j} = \Tr (\bP e_{i} \otimes e_{i}) (\bP e_{j} \otimes e_{j}) = \langle e_{i}, \bP e_{j} \rangle \langle \bP e_{i}, e_{j} \rangle = \bP_{ij}^{2}. 
    \end{equation*}
    The conclusion follows from Theorem \ref{thm:rank-eq-char}.
\end{proof}

%\begin{example}
%    It is not difficult to see that the set of rank one extreme correlation matrices is given by $\bP = [\sigma_{i}\sigma_{j}]_{i,j=1}^{n}$ for $(\sigma_{i})_{i=1}^{n} \in \{-1, 1\}^{n}$ for both $\bbF = \bbR$ and $\bbC$. 
%\end{example}

Note that Corollary \ref{thm:extrm-correlation} (2) is the equality case for an instance of the Hadamard rank inequality: 
\begin{equation*}
    \rank (\bA \odot \bB) \leq (\rank \bA)(\rank \bB)
\end{equation*}
where $\bA = \bB = \bP \in \sP$. 
It follows that for correlation matrices, the Hadamard rank inequality is an equality precisely when the involved matrix is an extreme correlation matrix. Notice that multiplying a positive semi-definite matrix $\bA$ on both sides by the diagonal matrix $\bB$ with the entries given by $\bB_{jj} = 1/\sqrt{\bA_{jj}}$ for $\bA_{jj} > 0$ and $0$ otherwise, makes it a correlation matrix $\bP = \bB\bA\bB$. Therefore, we can trivially extend this result to positive semi-definite matrices.
\begin{corollary}[Hadamard Rank Inequality]
    Let $\bA = [\bA_{ij}]_{i,j = 1}^{n} \in \bbF^{n \times n}$ be a positive semi-definite matrix. Then 
    \begin{equation}\label{eqn:hadamard}
        \rank \bA^{\odot 2} \leq (\rank \bA)^{2}
    \end{equation}
    with equality iff $\bA$ is an extreme point of the spectrahedron $\sC = \{\bP = [\bP_{ij}]_{i,j=1}^{n}\succeq \bzero: \bP_{jj} = \bA_{jj} \mbox{ for } j \in \{1, \dots , n\}\} \subset \bbC^{n \times n}$ (for $\bbF = \bbR$, we consider $\bbR^{n \times n}$ a subset of $\bbC^{n \times n}$). 
\end{corollary}

\begin{remark}
    Note that the above result is not of much relevance for real positive semi-definite matrices since these already satisfy $\rank \bA^{\odot 2} \leq \tfrac{1}{2}(\rank \bA)(\rank \bA + 1)$ and equality in (\ref{eqn:hadamard}) is achieved only for $\rank \bA = 1$.
\end{remark}

Although it is unclear how this result can be extended to the general case using the proposed techniques, it does reveal a rich structure underlying the equality case of the Hadamard rank inequality.

\section{Applications}
\label{sec:applications}

\subsection{Sensitivity Analysis of PCA with Incomplete Data}

In statistics, we often encounter covariance estimation problems with incomplete data. A particularly complicated version of the problem arising in longitudinal studies in medicine \cite{delaigle2021} can be stated as follows: 

\begin{quote}
    Let $X: [0, 1] \to \bbR$ be a (zero-mean) second order stochastic process. We are interested in estimating its covariance $K(s, t) = \bbE[X_{s}X_{t}]$ for $s, t \in I$ but we only have samples $X^{i}$ for $i = 1, \dots, N$ of the process which are observed over one of the overlapping intervals $I_{j}$ covering $I$ for $j = 1, \dots, r$. 
\end{quote}

Because the samples are incompletely observed, traditional methods of covariance estimation only allow us to estimate the restrictions $K|_{I_{j} \times I_{j}}$ of the covariance $K$ over the square $I_{j} \times I_{j} \subset I \times I$. In general, the covariance is not even uniquely determined even if we know the restrictions $K|_{I_{j} \times I_{j}}$ exactly. But we can construct a canonical completion $K_{\star}$ from these restrictions $K|_{I_{j} \times I_{j}}$ and it turns out that under certain intuitive and plausible conditions on the dependence structure of the process $X$, $K_{\star}$ is the true covariance of $X$ \cite{waghmare2022}.

However, when we are unsure about whether these conditions hold and thus, whether $K_{\star}$ is the true covariance of $X$, it is useful and reassuring to derive bounds on features of the covariance which are of interest to us. An example of such features are the eigenvalue-eigenvector pairs of the trace-class integral operator $\bP: L^{2}[0,1] \to L^{2}[0,1]$ induced by the true covariance $K: I \times I \to \bbR$ which are the basis for principal components analysis (PCA) \cite{hsing2015}. For instance, one might wish to bound the sum of the largest $q$ eigenvalues of $\bP$, which is a convex function of $\bP$ for $q \geq 1$. This will provide an indicator of the sensitivity of PCA to the assumptions that grant equality of the canonical completion and the true covariance.

We now illustrate how this can be done for the first eigenvalue of $\bP$ based on our results. 

For $j = 1, \dots, r$, let $\{e_{jk}\}_{k=1}^{\infty}$ be an orthonormal basis for $L^{2}(I_{j})$. Define $\tilde{e}_{jk}: I \to \bbR$ as $\tilde{e}_{jk}(t) = e_{jk}(t)$ for $t \in I_{j}$ and $0$ otherwise. We estimate the trace $\Tr \bP = \int \bbE[X_{t}^{2}]~dt$ and the inner products $\langle \bP \tilde{e}_{jk}, \tilde{e}_{jl} \rangle = \int_{I_{j}} \int_{I_{j}} K_{I_{j} \times I_{j}}(s, t) e_{jk}(s)e_{jl}(t) ~ds ~dt$ for every $j,l \leq p$ (for a suitably chosen $p \geq 1$) from the given samples. The true covariance operator $\bP$ can now be said to lie in a spectrahedron $\sC$ corresponding to $\{(\bI, \beta_{0})\} \cup \{(\bA_{jkl}, \beta_{jkl})\}_{j,k,l=1}^{r,p,p}$ where $\bA_{jkl} = \tilde{e}_{jk} \otimes \tilde{e}_{jl}$ and $\beta_{0}, \{\beta_{jkl}\}$ are the corresponding estimates.

To calculate an upper bound on the largest eigenvalue of the true covariance, we solve $\max_{\bP \in \sC} \lambda_{1}(\bP)$ where $\lambda_{1}(\bP)$ denotes the first eigenvalue of $\bP$. Note that the maximum exists (because the unit ball in trace norm $\|\cdot\|_{1}$ is weak$^{\ast}$ compact) and it is achieved at an extreme point. Corollary \ref{thm:bpbounds} tells us that the rank of an extreme point does not exceed $R = \sqrt{2rp^{2}+9/4} - 1/2$. It follows that we need to search only over positive operators $\bP$ of rank $ \leq R$ satisfying the linear constraints instead of \emph{all} such positive trace-class operators of arbitrary rank. In this way, we can derive an upper bound of the largest eigenvalue of the true covariance. %The spread between this eigenvalue and the largest eigenvalue of the canonical completion give us an idea of how sensitive our principal components analysis to the assumptions under which the canonical completion is the true covariance.

\subsection{Optimization on Quantum Mechanical States}

Consider a particle constrained to be on the unit interval $I = [0 ,1]$. The state-space associated with the particle is the space $\sH = L^{2}(I, \bbC)$ of square-integrable complex-valued functions on $I$. The position $X$ of the particle is a real-valued random variable and let's say that we are interested in characterizing the extreme states satisfying the moment constraints: $\bbE[X] = m_{1}$, $\bbE[X^{2}] = m_{2}$ and $\bbE[X^{3}] = m_{3}$. 

The state of the particle is represented by a density operator, that is, $\bP \in \sP$ satisfying $\Tr \bP = 1$. The moment constraints translate to $\{(\bA_{j}, \beta_{j})\}_{j=1}^{3}$ where $\bA_{j} \in \sS$ such that $\bA_{j} f(u) = u^{j} f(u)$ for $u \in [0, 1]$ and $\beta_{j} = m_{j}$. Together we have $4$ linear constraints, and by Corollary \ref{thm:bpbounds}, $\rank \bP \leq 2$ for every extreme point $\bP \in \sP$. It follows that $\bP = \alpha \psi \otimes \psi + (1-\alpha) \phi \otimes \phi$ for $0 \leq \alpha \leq 1$, $\psi,\phi \in \sH$ such that $\|\psi\|, \|\phi\| = 1$ and $\langle \psi, \phi \rangle = 0$. Imposing the linear constraints gives:
\begin{equation}\label{eqn:qm-constriants}
    \int u^{j} \Big[ \alpha|\psi(u)|^{2} + (1-\alpha)|\phi(u)|^{2}\Big] ~du = m_{j}
\end{equation}
for $j = 1, 2, 3$.

Now, if one wants to find a state that maximizes a strictly convex function over such states, the related optimization problem is much simpler. For example, the minimum of the von Neumann entropy $S(\bP) = - \Tr (\bP \log \bP)$ over such states can be found by solving the simpler optimization problem
\begin{equation*}
    \min_{\psi, \phi \in \sH}~ -\alpha\log(\alpha) - (1- \alpha)\log (1- \alpha) 
\end{equation*}
over orthonormal $\psi, \phi \in \sH$ satisfying (\ref{eqn:qm-constriants}). The study of entropy minimizing states is relevant to the study of maximally entangled states \cite{parthasarathy2005}.

\subsection*{Acknowledgments}

The authors would like to thank Prof. Johanna Ziegel, ETH Zurich and Prof. Paul Marriott, University of Waterloo for helpful discussions and encouragement.

\printbibliography
\end{document}